\documentclass[11pt]{article}   
   
\usepackage{latexsym,amsfonts,amsmath,epsfig,tabularx,amsthm}   

 \usepackage[all]{xy}

\usepackage{amsmath, amssymb, amsthm}
\usepackage[textwidth=2.8cm]{todonotes}
\usepackage{yfonts}
\usepackage{mathrsfs}
\usepackage{tikz,tikz-cd}

\usepackage{color} 

\usepackage{authblk}
     
\newtheorem{thm}{Theorem}   
\newtheorem{rem}[thm] {Remark} 
 
 \newtheorem{lem}[thm]{Lemma}   
\newtheorem{prop}[thm]{Proposition}

\newtheorem{defi}[thm]{Definition}

\newcommand{\Lip}{{\rm Lip}}

\newlength{\fixboxwidth}   
\title{\scshape\LARGE The unified version of mixing maps  between  non-void sets}

\begin{document}

\author[1]{\small Salam Adel Al-Bayati}
\author[2]{\small Akram Al-Sabbagh}
\author[3]{\small Manaf Adnan Saleh Saleh}

\affil[1]{\footnotesize Department of Mathematics and Computer Applications, College of Science, Al-Nahrain University, Baghdad, Iraq. 
Email: salamahmed@sc.nahrainuniv.edu.iq}
\affil[2]{\footnotesize Department of Mathematics and Computer Applications, College of Science, Al-Nahrain University, Baghdad, Iraq. Email: akj@sc.nahrainuniv.edu.iq}
\affil[3]{\footnotesize Department of Mathematics and Computer Applications, College of Science, Al-Nahrain University, Baghdad, Iraq. Email: mas@sc.nahrainuniv.edu.iq}

\maketitle   

\begin{abstract} 
The nonlinear concepts of mixed summable families and maps for the spaces that only non-void sets are developed. Several characterisations of the corresponding concepts are achieved and the proof for a general Pietsch Domination-type theorem is established. Furthermore, this work has presented plenty of composition and inclusion results between different classes of mappings in the abstract settings. Finally, a generalized notation of mixing maps and their characteristics are extended to a more general setting.
\end{abstract}    
\textbf{2010 AMS Subject Classification}. Primary 	47Jxx	; 	47Hxx; 	  Secondary 	54Cxx, 03E75. \\ 
\textbf{Key words and phrases}.  Hausdorff topological space, Set theory, Nonlinear mixing maps; Pietsch Domination Theorem; Composition Theorem.
\section{Notations and Preliminaries} 
We introduce notations that will be used in this article. Let $I$  be an index set. The letters  $\mathbb{R}^{+}$, $\mathbb{R}$, $\mathbb{N}$ stand for the set of all positive real numbers, the set of all real numbers and the set of all natural numbers, respectively, and $\mathbb{K}$ denotes the field of real or complex numbers. Let $A$, $B$, $C$, $C_{1}$ and $D$ be non-void sets  and $\mathcal{H}$ be a non-void family of mappings from $A$ into $B$. Let $E$, $F$ and $G$ be Banach spaces and the closed unit ball of a Banach space $E$ is denoted by $B_{E}$. The dual space of $E$ is  denoted by $E^{*}$. The letters $X$ and $Y$ stand for pointed metric spaces. A map $T$ from $X$ into $Y$ is called Lipschitz if there is a nonnegative constant $C$ such that $d_{Y}\left(T x_{1}, T x_{2}\right)\leq C\ d_{X}\left(x_{1}, x_{2}\right),$ for all $x_{1}$, $x_{2}$ in $X$. The smallest possible $C$ is the Lipschitz constant of $T$ denoted by $\Lip(T)$. The Banach space of real-valued Lipschitz maps defined on $X$ that send the special point $0$ to $0$ with the Lipschitz norm $\Lip(\cdot)$ will be denoted by $X^{\#}$. The space $X^{\#}$ is called the Lipschitz dual of $X$. Let $K$ and $W$ be compact Hausdorff topological spaces. The symbols $\textbf{W}(W)$, $\textbf{W}(B_{E^{*}})$ and $\textbf{W}(B_{X^{ \#}})$ stand for the set of all Borel probability measures defined on $W$, $B_{E^{*}}$ and $B_{X^{ \#}}$, respectively. The value of  $a$  at the element  $x$  is denoted by $\left\langle x, a\right\rangle$.
\section{Introduction}
The usual mathematical problems include nonlinear operators, occasionally influential on arbitrary sets with few (or none) algebraic structures; hence the extension of linear mechanisms to the nonlinear setting, besides its essential mathematical interest, is an important duty for potential applications. The full general version of maps (with no structure of the spaces included) would certainly be interesting for potential applications. G. Botelho, D. Pellegrino, and P. Rueda \cite{BPR10} defined the concept of \textbf{R}--\textbf{S}-abstract $p$-summing map as follows. Let $0<p<\infty$. A mapping $T\in\mathcal{H}$ is said to be \textbf{R}--\textbf{S}-abstract $p$-summing if there is a constant $\delta>0$ such that
\begin{equation}\label{lllllllllllllll}
\left[\sum\limits_{j=1}^{m}\textbf{S}(T, c_{j}, b_{j})^{p}\right]^{\frac{1}{p}}\leq \delta\cdot\sup\limits_{\varphi\in K}\left[\sum\limits_{j=1}^{m}\textbf{R}(\varphi, c_{j}, b_{j})^{p}\right]^{\frac{1}{p}}
\end{equation}
for all $c_{1}, \cdots, c_{m}\in C$, $b_{1}, \cdots, b_{m}\in G$ and $m\in\mathbb{N}$. The infimum of such constants $\delta$ is denoted by $\pi_{\textbf{R}\textbf{S}, p}(T)$.  They established a quite general Pietsch Domination-type Theorem under certain hypotheses on \textbf{R} and \textbf{S} as follows.

(\textbf{1}) For each $T\in\mathcal{H}$, there is $c_{0}\in C$ such that $\textbf{R}(\varphi, c_{0},b)=\textbf{S}(T, c_{0},b)=0$ for every $\varphi\in K$ and $b\in G$.

(\textbf{2}) The mapping $\textbf{R}_{c, b}: K\longrightarrow [0, \infty)$ defined by $\textbf{R}_{c, b}(\varphi)=\textbf{R}(c, b, \varphi)$ is continuous for every $c\in C$  and $b\in G$.

(\textbf{3}) It holds that $\textbf{R}(\varphi, c,  \eta\; b) \leq \eta\cdot\textbf{R}(\varphi, c, b)$ and $\eta\cdot \textbf{S}(T, c,  b) \leq \textbf{S}(T, c,  \eta\; b)$ for every $\varphi\in K$, $c\in C$, $0\leq\eta\leq 1$, $b\in G$ and $T\in\mathcal{H}$. 
\begin{thm}(see \cite{BPR10}.)\label{m2} 
If $\textbf{R}$ and $\textbf{S}$ satisfy Conditions (\textbf{1}), (\textbf{2}) and (\textbf{3}) and $0<p<\infty$, then $T\in\mathcal{H}$ is \textbf{R}--\textbf{S}-abstract $p$-summing map if and only if there are constant $\delta > 0$ and  Borel probability measure $\nu$ on $K$ such that $$\textbf{S}(T, c, b)\leq\delta\cdot\left[\int\limits_{K}\textbf{R}(c, b, \varphi)^{p} d\nu(\varphi)\right]^\frac{1}{p},$$
whenever $c\in C$ and $b\in G$. 
\end{thm}
Building upon the observation was made by M. Mendel and G. Schechtman that appears in \cite{FJ09}.  D. Pellegrino and J. Santos \cite{BPR11} defined the equivalent to Inequality (\ref{lllllllllllllll}) as follows. A mapping $T\in\mathcal{H}$ is said to be \textbf{R}--\textbf{S}-abstract $p$-summing if there is a constant $\delta>0$ such that
\begin{equation}\label{m3} 
\left[\sum\limits_{j=1}^{m}\lambda_{j}\;\textbf{S}(T, c_{j}, b_{j})^{p}\right]^{\frac{1}{p}}\leq \delta\cdot\sup\limits_{\varphi\in K}\left[\sum\limits_{j=1}^{m}\lambda_{j}\;\textbf{R}(\varphi, c_{j}, b_{j})^{p}\right]^{\frac{1}{p}}
\end{equation}
for all $c_{1}, \cdots, c_{m}\in C$, $b_{1}, \cdots, b_{m}\in G$, $\lambda_{1}, \cdots, \lambda_{m}\in\mathbb{R}^{+}$, and $m\in\mathbb{N}$. From Inequality (\ref{m3}) and invoking \cite [Theorem 2.1] {BPR10} they proved a general Pietsch Domination-type Theorem with no assumption on \textbf{S}
and just supposing that \textbf{R} satisfies Condition (\textbf{2}) as follows.
\begin{thm}(see \cite{BPR11}.)\label{m4}
If \textbf{R} satisfies Condition (\textbf{2}) and $0<p<\infty$, then $T\in\mathcal{H}$ be \textbf{R}--\textbf{S}-abstract $p$-summing map if and only if there are constant $\delta > 0$ and  Borel probability measure $\nu$ on $K$ such that $$\textbf{S}(T, c, b)\leq\delta\cdot\left[\int\limits_{K}\textbf{R}(c, b, \varphi)^{p} d\nu(\varphi)\right]^\frac{1}{p},$$
whenever $c\in C$ and $b\in G$. 
\end{thm}
D. Pellegrino, J. Santos and J. B. Seoane-Sep\'{u}lveda \cite{PSS12} defined  the concept of  $\textbf{R}_{1}, \ldots, \textbf{R}_{t}$--\textbf{S}-abstract $(p_{1} ,\ldots,p_{t})$-summing map as follows. Let $0<p<\infty$. A map $T$ from $A_1\times\cdots\times A_t$ into $B$ is called $\textbf{R}_{1}, \ldots, \textbf{R}_{t}$--\textbf{S}-abstract $(p_{1} ,\ldots,p_{t})$-summing if there is a constant $\delta\geq 0$ such that 
$$\left[\sum\limits_{j=1}^{m}\textbf{S}(T, c_{j}^{1}, \ldots, c_{j}^{r}, b_{j}^{1}, \ldots, b_{j}^{t})^{p}\right]^{\frac{1}{p}}\leq \delta\cdot\displaystyle\prod_{k=1}^{s}\sup\limits_{\varphi\in \textbf{K}_{k}}\left[\sum\limits_{j=1}^{m}\left|\textbf{R}_{k}(c_{j}^{1}, \ldots, c_{j}^{r},  b_{j}^{k}, \varphi)\right|^{p_{k}}\right]^{\frac{1}{p_{k}}}$$
for all  $c_{1}^{1}, \cdots, c_{m}^{r}\in  C_{s}$, $b_{1}^{1}, \cdots, b_{m}^{l}\in  G_{l}$, $m\in\mathbb{N}$ and $(s, l)\in\left\{1, \ldots, r\right\}\times\left\{1, \ldots, t\right\}$. They proved a quite general Pietsch Domination Theorem as follows.
\begin{thm} (see \cite{PSS12}.)
A map $T\in\mathcal{H}$ is \textbf{R}--\textbf{S}-abstract $p$-summing if and only if there are constant $\delta > 0$ and  Borel probability measure $\nu$ on $K$ such that $$\textbf{S}(T, c^{1}, \ldots, c^{r}, b^{1}, \ldots, b^{t})\leq \delta\cdot\displaystyle\prod_{j=1}^{t}\left(\int\limits_{K_{j}}\textbf{R}_{j}(c^{1}, \ldots, c^{r},  b^{j}, \varphi)^{p_{j}} d\nu_{j}(\varphi)\right)^{\frac{1}{p_{j}}}$$
for all $c^{l}\in C_{l}$, $l=1,\ldots, r$ and $b^{j}\in G_{j}$, with $j = 1,\ldots, t$. 
\end{thm} 
Several authors have investigated a special case version of the class of \textbf{R}--\textbf{S}-abstract $p$-summing maps starting with the seminal papers \cite{P67} (linear version) and \cite{FJ09} (Lipschitz version) and further explored applications in the nonlinear case can be found in \cite{CD11, C12, S16, ADM17, S18, S19}.

 We now describe the contents of this paper.  In Section \ref{III},  we modify Inequality (\ref{m3}) to construct the concept of \textbf{H}--\textbf{Q}-abstract $p$-summing map which is quite useful to prove the main results under certain assumptions in the forthcoming sections. In Section \ref{IIII}, we define the nonlinear version concept of \textbf{M}--mixed $(s;q)$-summable family in which the spaces are just arbitrary sets and establish an important characterization for this notion  under certain hypotheses in abstract settings. In Section \ref{IIIII}, we construct the concept of \textbf{H}--\textbf{M}-$((s; q), p)$-mixing maps between arbitrary sets and prove several characterizations. Afterwards  we show various composition and inclusion results between different classes of mappings in abstract setting  and prove a quite general of Pietsch Domination-type Theorem given in \cite {BPR11}. In Section \ref{IIIIII}, we prove how Proposition \ref{aaa} and Proposition \ref{abuleldgfgfhgfhgfddff} can be appealed in order to get some of the familiar characterizations that have appeared in the different generalizations of the concept of $(s; q)$-mixing operators. It is obvious to see that for suitable choices of $A$, $B$, $C$, $G$, $\mathcal{H}$, $K$, $W$, $\textbf{H}$, and $\textbf{M}$, for a mapping to belong to one of such classes of mixing maps is equivalent to be \textbf{H}--\textbf{M}-$((s; q), p)$-mixing map and the corresponding characterizations that hold for this class is nothing but Proposition \ref{aaa} and Proposition \ref{abuleldgfgfhgfhgfddff}. In Section \ref{IIIIIII}, we generalize a notion of mixing maps and show characterization for this notion to a more general setting. 
\section{\textbf{H}--\textbf{Q}-abstract $p$-summing maps} \label{III}
Let  $\mathcal{H}_{1}$ and $\mathcal{H}_{2}$ be non-void families of mappings from $B$ into $D$ and $A$ into $D$, respectively, and let
\begin{align}
&\textbf{Q}:\mathcal{H}\times A\times C\times G\longrightarrow \mathbb{R},\  \textbf{H}: A\times C\times G\times K\longrightarrow\mathbb{R}, \nonumber \\
&\textbf{Q}_{1}:\mathcal{H}_{1}\times B\times C_{1}\times G\longrightarrow \mathbb{R},\  \textbf{H}_{1}: B\times C_{1}\times G\times W\longrightarrow \mathbb{R}, \nonumber \\
&\textbf{Q}_{2}:\mathcal{H}_{2}\times A\times C\times G\longrightarrow \mathbb{R}  \nonumber
\end{align} 
be arbitrary maps satisfy the following conditions:

(\textbf{I})  The mapping $\textbf{H}_{a, c, g}: K\longrightarrow \mathbb{R}$ defined by $$\textbf{H}_{a, c, g}(\varphi)=\textbf{H}(a, c, g, \varphi)$$ is continuous for every $a\in A$, $c\in C$ and $g\in G$. 

(\textbf{II}) $\textbf{Q}_{2}(S\circ T, a, c, g) \leq\textbf{Q}_{1}(S, T a, c, g)$ for every $T\in\mathcal{H}$, $S\in\mathcal{H}_{1}$, $a\in A$, $c\in C$ and $g\in G$.   
\begin{defi}
Let $0<p<\infty$. A map $T\in\mathcal{H}$ is said to be \textbf{H}--\textbf{Q}-abstract $p$-summing if there is a constant $\delta>0$ such that
\begin{equation}\label{m1}
\left[\sum\limits_{j=1}^{m}\left|\sigma_j\right|^{p}\left|\textbf{Q}(T, a_{j}, c_{j}, g_{j})\right|^{p}\right]^{\frac{1}{p}}\leq \delta\cdot\sup\limits_{\varphi\in K}\left[\sum\limits_{j=1}^{m}\left|\sigma_j\right|^{p}\left|\textbf{H}(\varphi, a_{j}, c_{j}, g_{j})\right|^{p}\right]^{\frac{1}{p}}
\end{equation}
for all nonzero $\sigma_{1},\ldots, \sigma_{m}$ in $\mathbb{R}$, $a_{1},\ldots, a_{m}$ in $A$, $c_{1},\ldots, c_{m}$ in $C$, $g_{1},\ldots, g_{m}$ in $G$ and $m\in\mathbb{N}$. The infimum of such constants $\delta$ is denoted by $\pi_{\textbf{H}\textbf{Q}, p}(T)$. Let us denote by $\Pi_{p}^{\textbf{H}-\textbf{Q}}\left(A, B\right)$ the class of all \textbf{H}--\textbf{Q}-abstract $p$-summing maps from $A$ into $B$. 
\end{defi}
The proof of the next proposition is similar to the main implication in the nonlinear general Pietsch Domination-type Theorem of \cite [Theorem 3.1] {BPR11} and is therefore omitted.
\begin{prop}\label{msfgsfgfgf1}
Suppose that \textbf{Q} is  an arbitrary map and \textbf{H} satisfies Condition (\textbf{I}) and let $0<p<\infty$.  A map $T\in\mathcal{H}$ be \textbf{H}--\textbf{Q}-abstract $p$-summing if and only if there are constant $\delta > 0$ and  Borel probability measure $\nu$ on $K$ such that $$\left|\textbf{Q}(T, a, c, g)\right|\leq\delta\cdot\left[\int\limits_{K}\left|\textbf{H}(a, c, g, \varphi)\right|^{p} d\nu(\varphi)\right]^\frac{1}{p},$$
whenever $a\in A$, $c\in C$, and $g\in G$. 
\end{prop}
\section{\textbf{M}--mixed $(s;q)$-summable families} \label{IIII}
Throughout this section let $0<q\leq s\leq\infty$ and the index $r$ be determined by the equation  $\frac{1}{r}+\frac{1}{s}=\frac{1}{q}$. Let $\textbf{M}:\mathcal{H}\times A\times C\times G\times W\longrightarrow \mathbb{R}$ be an arbitrary map satisfy the following conditions: 

(\textbf{III})  The mapping $\textbf{M}_{T, a, c, g}: W\longrightarrow \mathbb{R}$ defined by  $$\textbf{M}_{T, a, c, g}(\psi)=\textbf{M}(T, a, c, g, \psi)$$ is continuous for every $T\in\mathcal{H}$, $a\in A$, $c\in C$ and $g\in G$. 

(\textbf{IV})  The mapping $\textbf{M}$ be a homogeneous of degree $1$ in the variable $W$ if $$\textbf{M}(T, a, c, g, \lambda\;\psi)=\lambda\;\textbf{M}(T, a, c, g, \psi).$$

(\textbf{V}) $\textbf{M}_{2}(S\circ T, a, c, g, \psi) \leq\textbf{M}_{1}(S, T a, c, g, \psi)$ for every $T\in\mathcal{H}$, $S\in\mathcal{H}_{1}$, $a\in A$, $c\in C$, $g\in G$ and $\psi\in W$.

(\textbf{VI})  $\textbf{H}_{1}(T a, c, g, \psi) \leq\textbf{M}(T, a, c, g, \psi)$ for every $T\in\mathcal{H}$, $a\in A$, $c\in C$, $g\in G$ and $\psi\in W$.

(\textbf{VII}) Let $1\leq s<\infty$ and let $\mu\in \textbf{W}(W)$ we consider the map $\textbf{\textit{J}}_{\mu}\in\mathfrak{P}_{s}^{\textbf{H}_{1}-\textbf{Q}_{1}}\left(B, D\right)$ with $\pi_{\textbf{H}_{1}\textbf{Q}_{1}, s}(\textbf{\textit{J}}_{\mu})\leq 1$ such that 
\begin{equation}\label{zabalawe}
\left[\int\limits_{W}\left|\textbf{M}(T, a, c, g, \psi)\right|^{s} d\mu(\psi)\right]^\frac{1}{s}\leq\left|\textbf{Q}_{2}(\textbf{\textit{J}}_{\mu}\circ T, a, c, g)\right|
\end{equation}
 for every $T\in\mathcal{H}$, $\textbf{\textit{J}}_{\mu}\in\mathcal{H}_{1}$, $a\in A$, $c\in C$ and $g\in G$. 
\begin{rem}
Condition (\textbf{VII}) can be applied in the following special cases.  
\begin{enumerate}
	\item We put $A:=E$, $B:=F$, $C=C_{1}:=\left\{1\right\}$, $D:=L_{s}(B_{F^{*}}, \mu)$, $G:=\mathbb{R}$, and $W:=B_{F^{*}}$. Let $\mathcal{H}$, $\mathcal{H}_{1}$ and $\mathcal{H}_{2}$ be non-void families of mappings from $E$ into $F$, $F$ into $L_{s}(B_{F^{*}}, \mu)$, and $E$ into $L_{s}(B_{F^{*}}, \mu)$, respectively. Let $T$ be an operator from $E$ into $F$  and let $\mu\in \textbf{W}(B_{F^{*}})$ we consider an operator  $\textbf{\textit{J}}_{\mu}$ from $F$ into $L_{s}(B_{F^{*}}, \mu)$ assigning to $y\in F$ the function $\textit{f}_{y}$ with $\textit{f}_{y}(y^{*}):=\left\langle y, y^{*}\right\rangle$, for more information see \cite {P78}.	We define the following maps.
\begin{align}
&\textbf{H}_{1}: F\times \left\{1\right\}\times \mathbb{R}\times B_{F^{*}}\longrightarrow \mathbb{R},\  \textbf{H}_{1}(y, 1, \sigma, y^{*})=\left\langle y, y^{*}\right\rangle, \nonumber \\
&\textbf{Q}_{1}: \mathcal{H}_{1}\times F\times \left\{1\right\}\times \mathbb{R}\longrightarrow \mathbb{R},\  \textbf{Q}_{1}(\textbf{\textit{J}}_{\mu}, y, 1, \sigma)=\left\|\textit{f}_{y} | L_{s}(B_{F^{*}}, \mu)\right\|, \nonumber \\
&\textbf{Q}_{2}:\mathcal{H}_{2}\times E\times \left\{1\right\}\times \mathbb{R}\longrightarrow \mathbb{R}  ,\  \textbf{Q}_{2}(\textbf{\textit{J}}_{\mu}\circ T, x, 1, \sigma)=\left\|\textit{f}_{Tx} | L_{s}(B_{F^{*}}, \mu)\right\|, \nonumber \\
&\textbf{M}:\mathcal{H}\times E\times \left\{1\right\}\times \mathbb{R}\times B_{F^{*}}\longrightarrow \mathbb{R},\  \textbf{M}(T, x, 1, \sigma, y^{*})= \left\langle T x, y^{*}\right\rangle. \nonumber 
\end{align} 	
With these choices we obtain $\textbf{\textit{J}}_{\mu}\in\mathfrak{P}_{s}^{\textbf{H}_{1}-\textbf{Q}_{1}}\left(F, L_{s}(B_{F^{*}}, \mu)\right)$ with $\pi_{\textbf{H}_{1}\textbf{Q}_{1}, s}(\textbf{\textit{J}}_{\mu})= 1$ and satisfy Inequality (\ref{zabalawe}).
\item We put $A:=X$, $B:=Y$, $C:=X$, $C_{1}:=Y$, $D:=L_{s}(B_{Y^{\#}}, \mu)$, $G:=\mathbb{R}$, and $W:=B_{Y^{\#}}$. Let $\mathcal{H}$, $\mathcal{H}_{1}$ and $\mathcal{H}_{2}$ be non-void families of mappings from $X$ into $Y$, $Y$ into $L_{s}(B_{Y^{\#}}, \mu)$, and $X$ into $L_{s}(B_{Y^{\#}}, \mu)$, respectively. Let $T$ be a Lipschitz map from $X$ into $Y$  and let $\mu\in \textbf{W}(B_{Y^{\#}})$ we consider Lipschitz map  $\textbf{\textit{J}}_{\mu}$ from $Y$ into $L_{s}(B_{Y^{\#}}, \mu)$ assigning to points $y_{1}$ and $y_{2}$ in $Y$ the function $\textit{f}_{(y_{1}, y_{2})}$ with $\textit{f}_{(y_{1}, y_{2})}(\tilde{g}):=\left\langle y_{1}, \tilde{g}\right\rangle -\left\langle y_{2}, \tilde{g}\right\rangle$, for more information see \cite {C12}.	We define the following maps.
\begin{align}
&\textbf{H}_{1}: Y\times Y\times \mathbb{R}\times B_{Y^{\#}}\longrightarrow \mathbb{R},\  \textbf{H}_{1}(y_{1}, y_{2}, \sigma, \tilde{g})=\left\langle y_{1}, \tilde{g}\right\rangle -\left\langle y_{2}, \tilde{g}\right\rangle, \nonumber \\
&\textbf{Q}_{1}: \mathcal{H}_{1}\times Y\times Y\times \mathbb{R}\longrightarrow \mathbb{R},\  \textbf{Q}_{1}(\textbf{\textit{J}}_{\mu}, y_{1}, y_{2}, \sigma)=\left\|\textit{f}_{(y_{1}, y_{2})} | L_{s}(B_{Y^{\#}}, \mu)\right\|, \nonumber \\
&\textbf{Q}_{2}:\mathcal{H}_{2}\times X\times X\times \mathbb{R}\longrightarrow \mathbb{R}  ,\  \textbf{Q}_{2}(\textbf{\textit{J}}_{\mu}\circ T, x_{1}, x_{2}, \sigma)=\left\|\textit{f}_{(Tx_{1}, Tx_{2})} | L_{s}(B_{Y^{\#}}, \mu)\right\|, \nonumber \\
&\textbf{M}:\mathcal{H}\times X\times X\times \mathbb{R}\times B_{Y^{\#}}\longrightarrow \mathbb{R},\  \textbf{M}(T, x_{1}, x_{2}, \sigma, \tilde{g})=\left\langle Tx_{1}, \tilde{g}\right\rangle -\left\langle Tx_{2}, \tilde{g}\right\rangle. \nonumber 
\end{align} 	
With these choices we obtain $\textbf{\textit{J}}_{\mu}\in\mathfrak{P}_{s}^{\textbf{H}_{1}-\textbf{Q}_{1}}\left(Y, L_{s}(B_{Y^{\#}}, \mu)\right)$ with $\pi_{\textbf{H}_{1}\textbf{Q}_{1}, s}(\textbf{\textit{J}}_{\mu})= 1$ and satisfy Inequality (\ref{zabalawe}).
\end{enumerate}
\end{rem}
The concept of \textbf{M}--mixed $(s;q)$-summable family can be constructed as follows. 
\begin{defi}\label{sadoon}
A family $\left((\sigma_{j}, T, a_{j}, c_{j}, g_{j}, \psi)\right)_{j\in I}\subset\mathbb{R}-\left\{0\right\}\times \mathcal{H}\times A\times C\times G\times W$ is called \textbf{M}--mixed $(s;q)$-summable   if there exists a nonzero family $(\tau_{j})_{j\in I}\in \ell_{r}(I)$ such that  $\sum\limits_{I} \left|\frac{\sigma_{j}}{\tau_{j}}\right|^{s} \left|\textbf{M}(T, a_{j}, c_{j}, g_{j}, \psi)\right|^{s}<\infty$. The class of all $\textbf{M}$-mixed $(s;q)$-summable families is denoted by $\mathfrak{M}_{(s;q)}^{\textbf{M}}(\mathbb{R}-\left\{0\right\}\times \mathcal{H}\times A\times  C\times G\times W, I)$. Moreover, for a family $\left((\sigma_{j}, T, a_{j}, c_{j}, g_{j}, \psi)\right)_{j\in I}\in\mathfrak{M}_{(s;q)}^{\textbf{M}}(\mathbb{R}-\left\{0\right\}\times \mathcal{H}\times A\times C\times G\times W, I)$ define 
\begin{equation}\label{three}
\mathfrak{m}_{(s;q)}^{\textbf{M}}\left((\sigma_{j}, T, a_{j}, c_{j}, g_{j}, \psi)\right)_{j\in I}=\inf\left[\sum\limits_{I}\left|\tau_j\right|^{r}\right]^{\frac{1}{r}}\sup\limits_{\psi\in W}\left[\sum\limits_{I}\left|\frac{\sigma_{j}}{\tau_{j}}\right|^{s} \left|\textbf{M}(T, a_{j}, c_{j}, g_{j}, \psi)\right|^{s}\right]^{\frac{1}{s}},
\end{equation}
where the infimum is taken over all  nonzero families $(\tau_{j})_{j\in I}\in \ell_{r}(I)$.
\end{defi}
The next result will be used in the next section. 
\begin{lem}\label{twotwo}
Let $\left((\sigma_{j}, T, a_{j}, c_{j}, g_{j}, \psi)\right)_{j\in I}$ be an arbitrary family in $\mathfrak{M}_{(s;q)}^{\textbf{M}}(\mathbb{R}-\left\{0\right\}\times \mathcal{H}\times A\times C\times G\times W, I)$. If $q=s$, then  $\mathfrak{m}_{(s;q)}^{\textbf{M}}\left((\sigma_{j}, T, a_{j}, c_{j}, g_{j}, \psi)\right)_{j\in I}=\sup\limits_{\psi\in W}\left[\sum\limits_{I}\left|\sigma_j\right|^{q} \left|\textbf{M}(T, a_{j}, c_{j}, g_{j}, \psi)\right|^{q}\right]^{\frac{1}{q}}$.
\end{lem}
\begin{proof}
Suppose that $\left((\sigma_{j}, T, a_{j}, c_{j}, g_{j}, \psi)\right)_{j\in I}\in\mathfrak{M}_{(s;q)}^{\textbf{M}}(\mathbb{R}-\left\{0\right\}\times \mathcal{H}\times A\times C\times G\times W, I)$, since $s=q$ we obtain $r=\infty$. By Definition \ref{sadoon}, there exists a family $(\tau_{j})_{j\in I}\in \ell_{\infty}(I)$  such that $\sum\limits_{I} \left|\frac{\sigma_{j}}{\tau_{j}}\right|^{q} \left|\textbf{M}(T, a_{j}, c_{j}, g_{j}, \psi)\right|^{q}<\infty$. Now we have
$$\sup\limits_{\psi\in W}\left[\sum\limits_{I} \left|\sigma_j\right|^{q} \left|\textbf{M}(T, a_{j}, c_{j}, g_{j}, \psi)\right|^{q}\right]^{\frac{1}{q}}\leq\left\|(\tau_{j})_{j\in I}\Big|\ell_{\infty}(I)\right\|\cdot\sup\limits_{\psi\in W}\left[\sum\limits_{I}\left|\frac{\sigma_{j}}{\tau_{j}}\right|^{s} \left|\textbf{M}(T, a_{j}, c_{j}, g_{j}, \psi)\right|^{s}\right]^{\frac{1}{s}}.$$   
Hence $\mathfrak{m}_{(s;q)}^{\textbf{M}}\left((\sigma_{j}, T, a_{j}, c_{j}, g_{j}, \psi)\right)_{j\in I}\geq\sup\limits_{\psi\in W}\left[\sum\limits_{I} \left|\sigma_j\right|^{q} \left|\textbf{M}(T, a_{j}, c_{j}, g_{j}, \psi)\right|^{q}\right]^{\frac{1}{q}}$. For the other direction, choose $(\tau_{j})_{j\in I}=1\in\ell_{\infty}(I)$. Then $\left\|(\tau_{j})_{j\in I}\Big|\ell_{\infty}(I)\right\|=1$ and $$\sup\limits_{\psi\in W}\left[\sum\limits_{I} \left|\sigma_j\right|^{s} \left|\textbf{M}(T, a_{j}, c_{j}, g_{j}, \psi)\right|^{s}\right]^{\frac{1}{s}}=\sup\limits_{\psi\in W}\left[\sum\limits_{I}\left|\frac{\sigma_{j}}{\tau_{j}}\right|^{s} \left|\textbf{M}(T, a_{j}, c_{j}, g_{j}, \psi)\right|^{s}\right]^{\frac{1}{s}}.$$ Then
\begin{align} 
\mathfrak{m}_{(s;q)}^{\textbf{M}}\left((\sigma_{j}, T, a_{j}, c_{j}, g_{j}, \psi)\right)_{j\in I}&=\inf\left\|(\tau_{j})_{j\in I}\Big|\ell_{\infty}(I)\right\| \sup\limits_{\psi\in W}\left[\sum\limits_{I}\left|\frac{\sigma_{j}}{\tau_{j}}\right|^{q} \left|\textbf{M}(T, a_{j}, c_{j}, g_{j}, \psi)\right|^{q}\right]^{\frac{1}{q}} \nonumber \\
&\leq\sup\limits_{\psi\in W}\left[\sum\limits_{I}\left|\sigma_j\right|^{q} \left|\textbf{M}(T, a_{j}, c_{j}, g_{j}, \psi)\right|^{q}\right]^{\frac{1}{q}}. \nonumber 
\end{align} 
\end{proof}
Inspired by analogous result in the linear theory of A. Pietsch \cite [Theorem 16.4.3] {P78} and the similar proof of \cite [Proposition 4.2] {C12} we give an important characterization of \textbf{M}--mixed $(s;q)$-summable family. 
\begin{prop} \label{two}  
Let $0<q<s<\infty$ and let \textbf{M} satisfies Condition (\textbf{III}). A family $\left((\sigma_{j}, T, a_{j},  c_{j}, g_{j}, \psi)\right)_{j\in I}\in\mathfrak{M}_{(s;q)}^{\textbf{M}}(\mathbb{R}-\left\{0\right\}\times \mathcal{H}\times A\times C\times G\times W, I)$ be \textbf{M}--mixed $(s;q)$-summable   if and only if 
\begin{equation}\label{twofive}
\left[\sum\limits_{I}\left|\sigma_j\right|^{q}\left(\int\limits_{W} \left|\textbf{M}(T, a_{j}, c_{j}, g_{j}, \psi)\right|^{s} d\mu(\psi)\right)^\frac{q}{s}\right]^{\frac{1}{q}}<\infty
\end{equation}
for every $\mu\in \textbf{W}(W)$. In this case 
$$\sup\limits_{\mu\in \textbf{W}(W)}\left[\sum\limits_{I}\left|\sigma_j\right|^{q}\left(\int\limits_{W} \left|\textbf{M}(T, a_{j}, c_{j}, g_{j}, \psi)\right|^{s} d\mu(\psi)\right)^\frac{q}{s}\right]^{\frac{1}{q}}=\mathfrak{m}_{(s;q)}^{\textbf{M}}\left((\sigma_{j}, T, a_{j}, c_{j}, g_{j}, \psi)\right)_{j\in I}.$$
\end{prop} 
\begin{proof}
Suppose that the family $\left((\sigma_{j}, T, a_{j}, c_{j}, g_{j}, \psi)\right)_{j\in I}$ satisfies (\ref{twofive}). Define a number $N$ as follows.
$$N=\sup\limits_{\mu\in \textbf{W}(W)}\left[\sum\limits_{I}\left|\sigma_j\right|^{q}\left(\int\limits_{W} \left|\textbf{M}(T, a_{j}, c_{j}, g_{j}, \psi)\right|^{s} d\mu(\psi)\right)^\frac{q}{s}\right]^{\frac{1}{q}}.$$
Then $N$ is finite. Put $u=\frac{r}{q}$ and $v=\frac{s}{q}$. Then $\frac{1}{u}+\frac{1}{v}=1$. We now consider the compact, convex subset  $$Z=\left\{\xi=\left(\xi_j\right)_{j\in I}: \sum\limits_{I}\xi_j^{u}\leq N^{q}\ and \; \xi_j\geq 0\right\}$$
of $\ell_{u}(I)$. Notice that the equation $$\phi(\xi)=\sum\limits_{I}\left|\sigma_j\right|^{s}\left(\xi_j +\epsilon\right)^{- v}\cdot\int\limits_{W} \left|\textbf{M}(T, a_{j}, c_{j}, g_{j}, \psi)\right|^{s} d\mu(\psi),$$
where $\mu\in \textbf{W}(W)$, $\epsilon > 0$, defines a continuous convex function $\phi$ on $Z$. Take the special family $\left(\xi_j\right)_{j\in\mathbb{N}}$ with $$\xi_j=\left(\int\limits_{W}\left|\sigma_j\right|^{s} \left|\textbf{M}(T, a_{j}, c_{j}, g_{j}, \psi)\right|^{s} d\mu(\psi)\right)^{\frac{1}{u\cdot v}}.$$
Then $\xi\in Z$ and $\phi(\xi)\leq N^{q}$. Since the collection $\mathfrak{Q}$ of all functions $\phi$ obtained in this way is concave, by \cite [Lemma E.4.2] {P78} we can find $\xi^{0}\in  Z$ such that $\phi(\xi^{0})\leq N^{q}$ for all $\phi\in\mathfrak{Q}$. In particular, considering the Dirac measure $\delta_{\psi}$ at   $\psi\in W$ we obtain 
$$\sum\limits_{I}\left|\sigma_j\right|^{s}\left(\xi_{j}^{0} +\epsilon\right)^{- v} \left|\textbf{M}(T, a_{j}, c_{j}, g_{j}, \psi)\right|^{s}\leq N^{q}.$$
Set $\tau_j(\epsilon)=\left(\xi^{0}_{j}+\epsilon\right)^{\frac{1}{q}}$. Then
\begin{equation} 
\left[\sum\limits_{I}\left|\tau_j\right|^{r}\right]^{\frac{1}{r}}=\lim_{\epsilon\rightarrow 0 ^{+}}\left[\sum\limits_{I}\left|\tau_{j}(\epsilon)\right|^{r}\right]^{\frac{1}{r}}=\left[\sum\limits_{I}(\xi^{0}_{j})^{\frac{r}{q}}\right]^{\frac{1}{r}}=\left[\sum\limits_{I}(\xi^{0}_{j})^{u}\right]^{\frac{1}{r}}\leq N^{\frac{q}{r}}= N^{\frac{1}{u}},
\end{equation} 
and for $\psi\in W$
\begin{align} 
\left[\sum\limits_{I}\left|\frac{\sigma_j}{\tau_{j}}\right|^{s} \left|\textbf{M}(T, a_{j}, c_{j}, g_{j}, \psi)\right|^{s}\right]^{\frac{1}{s}}&=\lim_{\epsilon\rightarrow 0 ^{+}}\left[\sum\limits_{I}\left|\frac{\sigma_j}{\tau_{j}(\epsilon)}\right|^{s} \left|\textbf{M}(T, a_{j}, c_{j}, g_{j}, \psi)\right|^{s}\right]^{\frac{1}{s}}    \nonumber \\
&=\lim_{\epsilon\rightarrow 0 ^{+}}\left[\sum\limits_{I}\left|\frac{\sigma_j^{s}}{(\xi^{0}_{j} + \epsilon)^{v}}\right| \; \left|\textbf{M}(T, a_{j}, c_{j}, g_{j}, \psi)\right|^{s}\right]^{\frac{1}{s}}\leq N^{\frac{q}{s}}= N^{\frac{1}{v}}.   \nonumber
\end{align}
Hence $$\left[\sum\limits_{I}\left|\tau_j\right|^{r}\right]^{\frac{1}{r}} \left[\sum\limits_{I}\left|\frac{\sigma_j}{\tau_{j}}\right|^{s} \left|\textbf{M}(T, a_{j}, c_{j}, g_{j}, \psi)\right|^{s}\right]^{\frac{1}{s}}\leq N.$$ This proves the necessity of  the above condition. Conversely, suppose that a family $\left((\sigma_{j}, T, a_{j}, c_{j}, g_{j}, \psi)\right)_{j\in I}$ is $\textbf{M}$-mixed $(s;q)$-summable. Take any family $(\tau_{j})_{j\in I}\in \ell_{r}(I)$ such that $\sum\limits_{I}\left|\frac{\sigma_{j}}{\tau_{j}}\right|^{s} \left|\textbf{M}(T, a_{j}, c_{j}, g_{j}, \psi)\right|^{s}<\infty$. Applying H\"older inequality we obtain 
\begin{align} 
\left[\sum\limits_{I}\left|\sigma_j\right|^{q}\left(\int\limits_{W} \left|\textbf{M}(T, a_{j}, c_{j}, g_{j}, \psi)\right|^{s} d\mu(\psi)\right)^\frac{q}{s}\right]^{\frac{1}{q}}&\leq\left[\sum\limits_{I}\left|\tau_j\right|^{r}\right]^{\frac{1}{r}}\sup\limits_{\psi\in W}\left[\sum\limits_{I}\left|\frac{\sigma_{j}}{\tau_{j}} \right|^{s} \left|\textbf{M}(T, a_{j}, c_{j}, g_{j}, \psi)\right|^{s}\right]^{\frac{1}{s}}  \nonumber
\end{align}
whenever $\mu\in \textbf{W}(W)$. This proves the sufficiency of the above condition.
\end{proof}
\section{\textbf{H}--\textbf{M}-$\left((s, q), p\right)$-mixing maps} \label{IIIII} 
Throughout this section, we assume that $V$ be a vector space over the field $\mathbb{K}$ and let $\mathcal{P}=\left(\textbf{P}_{\iota}\right)_{\iota =1}^{r}$ be a finite  family  of  semi-norms on  $V$. The topology induced by a  finite family  of  semi-norms on  $V$ is denoted by $\mathcal{P}$-topology on $V$. If $\textbf{P}_{\iota}$ $\left(1\leq\iota\leq r\right)$ be a finite family of semi-norms on a vector space $V$, then the function $\textbf{P}=\max \textbf{P}_{\iota}$ defined by
$$\textbf{P}(v)=\max\limits_{1\leq\iota\leq r}\textbf{P}_{\iota}(v)$$
be also a semi-norm on $V$. 

From  \cite [Proposition 2.16] {GA66}  it is known that the topology associated with the semi-norm $\textbf{P}$ is identical with the $\mathcal{P}$-topology on $V$. Suppose that $\mathfrak{B}_{\textbf{P}}$ be a compact unit $\textbf{P}$-ball defined as follows. $$\mathfrak{B}_{\textbf{P}}=\left\{v\in V : \textbf{P}(v)\leq 1\right\}.$$
The concept of \textbf{H}--\textbf{M}-$\left((s; q), p\right)$-mixing map can be constructed as follows.
\begin{defi}\label{hmdy}
Let $0<q\leq s\leq\infty$ and $p\leq q$. A map $T$ from $A$ into $B$ is called \textbf{H}--\textbf{M}-$\left((s; q), p\right)$-mixing if there is a constant $\delta\geq 0$ such that 
$$\mathfrak{m}_{(s;q)}^{\textbf{M}}\left((\sigma_{j}, T, a_{j}, c_{j}, g_{j}, \psi)\right)_{j=1}^{m}\leq \delta\cdot\sup\limits_{\varphi\in K}\left[\sum\limits_{j=1}^{m} \left|\sigma_j\right|^{p} \left|\textbf{H}(a_{j}, c_{j}, g_{j}, \varphi)\right|^{p}\right]^{\frac{1}{p}}$$
for all nonzero $\sigma_{1}, \cdots,\sigma_{m}\in\mathbb{R}$, $a_{1}, \cdots,a_{m}\in A$, $c_{1}, \cdots,c_{m}\in C$, $g_{1}, \cdots,g_{m}\in G$ and $m\in\mathbb{N}$. The infimum of such constants $\delta$ is denoted by $\textbf{H}\textbf{M}_{\left((s; q), p\right)}(T)$. Let us denote by $\mathfrak{M}_{\left((s; q), p\right)}^{\textbf{H}-\textbf{M}}\left(A, B\right)$ the class of all \textbf{H}--\textbf{M}-$\left((s; q), p\right)$-mixing maps from $A$ into $B$.
\end{defi}
Inspired by analogous result in the linear theory of A. Pietsch \cite [Theorem 20.1.4] {P78} and the similar proof of \cite [Theorem 4.1] {C12} we give the following characterization of \textbf{H}--\textbf{M}-$\left((s; q), p\right)$-mixing map. 
\begin{prop}\label{aaa}
Let $0<q< s<\infty$ and $p\leq q$ and let \textbf{H} and \textbf{M} satisfy Conditions (\textbf{I}), (\textbf{III}) and (\textbf{IV}), respectively. A map $T$ from $A$ into $B$ is \textbf{H}--\textbf{M}-$\left((s; q), p\right)$-mixing if and only if there is a constant $\delta\geq 0$ such that 
\begin{align}\label{pppppppppppppp}
\Bigg[\sum\limits_{j=1}^{m}\left|\sigma_j\right|^{q}\bigg[\sum\limits_{k=1}^{n}\left|\textbf{M}(T, a_{j}, c_{j}, g_{j}, v_{k})\right|^{s}\bigg]^\frac{q}{s}\Bigg]^{\frac{1}{q}}&\leq\delta\cdot\sup\limits_{\varphi\in K}\left[\sum\limits_{j=1}^{m}\left|\sigma_j\right|^{p} \left|\textbf{H}(a_{j}, c_{j}, g_{j}, \varphi)\right|^{p}\right]^{\frac{1}{p}} \nonumber \\
&\cdot\left[\sum\limits_{k=1}^{n} \textbf{P}(v_{k})^{s}\right]^{\frac{1}{s}}  
\end{align} 
for every nonzero $\sigma_{1}, \cdots,\sigma_{m}\in\mathbb{R}$, $a_{1}, \cdots, a_{m}\in A$, $c_{1}, \cdots, c_{m}\in C$, $g_{1}, \cdots, g_{m}\in G$; $v_1,\cdots, v_n\in V$ and $m,n\in\mathbb{N}$. Moreover 
$$\textbf{H}\textbf{M}_{\left((s; q), p\right)}(T)=\inf \delta.$$
\end{prop}
\begin{proof} 
Assume that $T$ is \textbf{H}--\textbf{M}-$\left((s, q), p\right)$-mixing map. Consider $v_1,\cdot\cdot\cdot, v_n\in V$ and define the discrete probability  $\mu=\sum\limits_{k=1}^{n} t_k \delta_k$, where $t_k=\textbf{P}(v_{k})^{s}\cdot\left[\sum\limits_{h=1}^{n} \textbf{P}(v_{h})^{s}\right]^{-1}$ and $\delta_k$ denotes the Dirac measure at $y_k =\frac{v_k}{\textbf{P}(v_{k})}\in \mathfrak{B}_{\textbf{P}}$; $k=1,\cdot\cdot\cdot, n$. Then $\mu\in \textbf{W}(\mathfrak{B}_{\textbf{P}})$. For $\sigma_{1}, \cdots,\sigma_{m}\in\mathbb{R}$, $a_{1}, \cdots, a_{m}\in A$, $c_{1}, \cdots,c_{m}\in C$, and $g_{1}, \cdots,g_{m}\in G$ we obtain from Proposition \ref{two} that
\begin{align}
\Bigg[\sum\limits_{j=1}^{m}\left|\sigma_j\right|^{q}\bigg[\sum\limits_{k=1}^{n}\left|\textbf{M}(T, a_{j}, c_{j}, g_{j}, v_{k})\right|^{s}\bigg]^\frac{q}{s}\Bigg]^{\frac{1}{q}}&=\Bigg[\sum\limits_{j=1}^{m}\left|\sigma_j\right|^{q}\bigg[\int\limits_{\mathfrak{B}_{\textbf{P}}}\left|\textbf{M}(T, a_{j}, c_{j}, g_{j}, v)\right|^{s} d\mu(v)\bigg]^\frac{q}{s}\Bigg]^{\frac{1}{q}}\cdot\left[\sum\limits_{k=1}^{n} \textbf{P}(v_{k})^{s}\right]^{\frac{1}{s}}  \nonumber \\ 
&\leq\mathfrak{m}_{(s;q)}^{\textbf{M}}\left((\sigma_{j}, T, a_{j}, c_{j}, g_{j}, v)\right)_{j=1}^{m}\cdot\left[\sum\limits_{k=1}^{n} \textbf{P}(v_{k})^{s}\right]^{\frac{1}{s}}\nonumber \\
&\leq\textbf{H}\textbf{M}_{\left((s; q), p\right)}(T)\cdot\sup\limits_{\varphi\in K}\left[\sum\limits_{j=1}^{m} \left|\sigma_j\right|^{p} \left|\textbf{H}(a_{j}, c_{j}, g_{j}, \varphi)\right|^{p}\right]^{\frac{1}{p}} \nonumber \\
&\cdot\left[\sum\limits_{k=1}^{n} \textbf{P}(v_{k})^{s}\right]^{\frac{1}{s}}. \nonumber
\end{align} 
To show the converse, notice that (\ref{pppppppppppppp}) means 
\begin{equation}\label{one}
\Bigg[\sum\limits_{j=1}^{m}\left|\sigma_j\right|^{q}\bigg[\int\limits_{\mathfrak{B}_{\textbf{P}}}\left|\textbf{M}(T, a_{j}, c_{j}, g_{j}, v)\right|^{s} d\mu(v)\bigg]^\frac{q}{s}\Bigg]^{\frac{1}{q}}\leq \delta\cdot\sup\limits_{\varphi\in K}\left[\sum\limits_{j=1}^{m}\left|\sigma_j\right|^{p} \left|\textbf{H}(a_{j}, c_{j}, g_{j}, \varphi)\right|^{p}\right]^{\frac{1}{p}}
\end{equation}
for every discrete probability measure $\mu$ on $\mathfrak{B}_{\textbf{P}}$ and $\sigma_{1}, \cdots,\sigma_{m}\in\mathbb{R}$, $a_{1}, \cdots, a_{m}\in A$, $c_{1}, \cdots,c_{m}\in C$, and $g_{1}, \cdots,g_{m}\in G$. Since the set of all finitely supported probability measures on $\mathfrak{B}_{\textbf{P}}$ be $\sigma\left(C(\mathfrak{B}_{\textbf{P}})^{\ast},C(\mathfrak{B}_{\textbf{P}})\right)$-dense in the set of all probability measures on $\mathfrak{B}_{\textbf{P}}$, it follows that (\ref{one}) holds for all probability measures $\mu$ on $\mathfrak{B}_{\textbf{P}}$ and $\sigma_{1}, \cdots,\sigma_{m}\in\mathbb{R}$, $a_{1}, \cdots, a_{m}\in A$, $c_{1}, \cdots,c_{m}\in C$, $g_{1}, \cdots,g_{m}\in G$. Taking the supremum over $\mu\in W(\mathfrak{B}_{\textbf{P}})$ on the left side of (\ref{one}) and using Proposition \ref{two}, we obtain $$\mathfrak{m}_{(s;q)}^{\textbf{M}}\left((\sigma_{j}, T, a_{j}, c_{j}, g_{j}, v)\right)_{j=1}^{m}\leq \delta\cdot\sup\limits_{\varphi\in K}\left[\sum\limits_{j=1}^{m}\left|\sigma_j\right|^{p} \left|\textbf{H}(a_{j}, c_{j}, g_{j}, \varphi)\right|^{p}\right]^{\frac{1}{p}}.$$
\end{proof}
The following multiplication formula represents the main-point of the theory of \textbf{H}--\textbf{M}-$\left((s; q), q\right)$-mixing maps and it is somewhat inspired by analogous result in the linear theory.
\begin{prop}\label{abuleldfddff}  
Let $0<q\leq s\leq\infty$. If the maps $\textbf{Q}_{1}$, $\textbf{Q}_{2}$, $\textbf{H}_{1}$ and $\textbf{M}$ satisfy Conditions  (\textbf{II}) and (\textbf{VI}), respectively, then $$\left[\mathfrak{P}_{s}^{\textbf{H}_{1}-\textbf{Q}_{1}}\left(B, D\right), {(\textbf{H}_{1}\textbf{Q}_{1})}_{s}\right]\circ\left[\mathfrak{M}_{\left((s; q), q\right)}^{\textbf{H}-\textbf{M}}\left(A, B\right), \textbf{H}\textbf{M}_{\left((s; q), q\right)}\right]\subseteq\left[\mathfrak{P}_{q}^{\textbf{H}-\textbf{Q}_{2}}\left(A, C\right), {(\textbf{H}\textbf{Q}_{2})}_{q}\right].$$
\end{prop}
\begin{proof}
Suppose that $S\in\mathfrak{P}_{s}^{\textbf{H}_{1}-\textbf{Q}_{1}}\left(B, D\right)$ and $T\in\mathfrak{M}_{\left((s, q), q\right)}^{\textbf{H}-\textbf{M}}\left(A, B\right)$. Given $\sigma_{1},\ldots, \sigma_{m}$ in $\mathbb{R}$, $a_{1},\ldots, a_{m}$ in $A$, $c_{1},\ldots, c_{m}$, $g_{1},\ldots, g_{m}$ in $G$, and $\epsilon>0$, we have
\begin{align}
\left[\sum\limits_{j=1}^{m}\left|\tau_j\right|^{r}\right]^{\frac{1}{r}} &\sup\limits_{\psi\in W}\left[\sum\limits_{j=1}^{m}\left|\frac{\sigma_{j}}{\tau_{j}}\right|^{s} \left|\textbf{M}(T, a_{j}, c_{j}, g_{j}, \psi)\right|^{s}\right]^{\frac{1}{s}}\leq (1+\epsilon)\cdot\mathfrak{m}_{(s;q)}^{\textbf{M}}\left((\sigma_{j}, T, a_{j}, c_{j}, g_{j}, \psi)\right)_{j=1}^{m} \nonumber \\
&\leq (1+\epsilon)\cdot\textbf{H}\textbf{M}_{\left((s; q), q\right)}(T)\cdot\sup\limits_{\varphi\in K}\left[\sum\limits_{j=1}^{m}\left|\sigma_j\right|^{q} \left|\textbf{H}(a_{j}, c_{j}, g_{j}, \varphi)\right|^{q}\right]^{\frac{1}{q}}.  \nonumber 
\end{align} 
We now notice from
$$\left[\sum\limits_{j=1}^{m}\left|\sigma_j\right|^{s} \left|\textbf{Q}_{1}(S, b_{j}, c_{j}, g_{j})\right|^{s}\right]^{\frac{1}{s}}\leq{(\textbf{H}_{1}\textbf{Q}_{1})}_{s}(S)\sup\limits_{\psi\in W}\left[\sum\limits_{j=1}^{m}\left|\sigma_j\right|^{s} \left|\textbf{H}_{1}(b_{j}, c_{j}, g_{j}, \psi)\right|^{s}\right]^{\frac{1}{s}}$$
that
$$\left[\sum\limits_{j=1}^{m}\left|\sigma_j\right|^{s} \left|\textbf{Q}_{1}(S, T a_{j}, c_{j}, g_{j})\right|^{s}\right]^{\frac{1}{s}}\leq{(\textbf{H}_{1}\textbf{Q}_{1})}_{s}(S)\sup\limits_{\psi\in W}\left[\sum\limits_{j=1}^{m}\left|\sigma_j\right|^{s} \left|\textbf{H}_{1}(T a_{j}, c_{j}, g_{j}, \psi)\right|^{s}\right]^{\frac{1}{s}}.$$ 
By applying H\"older inequality and Conditions  (\textbf{II}) and (\textbf{VI}) we obtain 
\begin{align}
\left[\sum\limits_{j=1}^{m} \left|\sigma_j\right|^{q} \left|\textbf{Q}_{2}(S\circ T, a_{j}, c_{j}, g_{j})\right|^{q}\right]^{\frac{1}{q}}&\leq\left[\sum\limits_{j=1}^{m} \left|\sigma_j\right|^{q} \left|\textbf{Q}_{1}(S, T a_{j}, c_{j}, g_{j})\right|^{q}\right]^{\frac{1}{q}}  \nonumber \\
&\leq\left[\sum\limits_{j=1}^{m}\left|\tau_j\right|^{r}\right]^{\frac{1}{r}} \left[\sum\limits_{j=1}^{m}\left|\frac{\sigma_j}{\tau_j}\right|^{s} \left|\textbf{Q}_{1}(S, T a_{j}, c_{j}, g_{j})\right|^{s}\right]^{\frac{1}{s}}  \nonumber \\
&\leq{(\textbf{H}_{1}\textbf{Q}_{1})}_{s}(S)\left[\sum\limits_{j=1}^{m}\left|\tau_j\right|^{r}\right]^{\frac{1}{r}} \sup\limits_{\psi\in W}\left[\sum\limits_{j=1}^{m}\left|\frac{\sigma_j}{\tau_j}\right|^{s} \left|\textbf{H}_{1}(T a_{j}, c_{j}, g_{j}, \psi)\right|^{s}\right]^{\frac{1}{s}} \nonumber \\
&\leq{(\textbf{H}_{1}\textbf{Q}_{1})}_{s}(S)\left[\sum\limits_{j=1}^{m}\left|\tau_j\right|^{r}\right]^{\frac{1}{r}} \sup\limits_{\psi\in W}\left[\sum\limits_{j=1}^{m}\left|\frac{\sigma_j}{\tau_j}\right|^{s} \left|\textbf{M}(T, a_{j}, c_{j}, g_{j}, \psi)\right|^{s}\right]^{\frac{1}{s}} \nonumber\\
&\leq (1+\epsilon)\cdot {(\textbf{H}_{1}\textbf{Q}_{1})}_{s}(S)\;\textbf{H}\textbf{M}_{\left((s; q), q\right)}(T)\sup\limits_{\varphi\in K}\left[\sum\limits_{j=1}^{m} \left|\sigma_j\right|^{q} \left|\textbf{H}(a_{j}, c_{j}, g_{j}, \varphi)\right|^{q}\right]^{\frac{1}{q}}.  \nonumber
\end{align}
Hence $S\circ T\in\mathfrak{P}_{q}^{\textbf{H}-\textbf{Q}_{2}}\left(A, D\right)$ with ${(\textbf{H}\textbf{Q}_{2})}_{q}(S\circ T)\leq{(\textbf{H}_{1}\textbf{Q}_{1})}_{s}(S)\cdot\textbf{H}\textbf{M}_{\left((s; q), q\right)}(T)$.
\end{proof} 
The following characterization is a quite general of unified Pietsch domination theorem given in \cite [Theorem 3.1] {BPR11}.
\begin{prop}\label{abuleldgfgfhgfhgfddff}  
Let $0<q\leq s\leq\infty$ and let the maps $\textbf{M}$, $\textbf{H}$, $\textbf{\textit{J}}_{\mu}$ and $\textbf{Q}_{2}$   satisfy Conditions (\textbf{I}), (\textbf{III}),  and (\textbf{VII}), respectively. A map $T$ is $\textbf{H}$--$\textbf{M}$-$\left((s; q), q\right)$-mixing if and only if there exists a constant $\delta\geq 0$ such that for any probability measure $\mu$ on $W$  there exists a probability measure $\nu$ on $K$ such that $$\left[\int\limits_{W}\left|\textbf{M}(T, a, c, g, \psi)\right|^{s} d\mu(\psi)\right]^\frac{1}{s}\leq\delta\cdot\left[\int\limits_{K}\left|\textbf{H}(a, c, g, \varphi)\right|^{q} d\nu(\varphi)\right]^\frac{1}{q},$$
whenever $a\in A$, $c\in C$ and $g\in G$. Moreover $\textbf{H}\textbf{M}_{\left((s; q), q\right)}(T)=\inf \delta$.
\end{prop} 
\begin{proof}
Let $\sigma$ be an arbitrary nonzero sequence in $\mathbb{R}$. By the assumptions, we have
\begin{align}\label{five}
\left[\sum\limits_{j=1}^{m}\left|\sigma_j\right|^{q}\left[\int\limits_{W}\left|\textbf{M}(T, a_j, c_j, g_j, \psi)\right|^{s} d\mu(\psi)\right]^{\frac{q}{s}}\right]^\frac{1}{q}&\leq\delta\cdot\left[\sum\limits_{j=1}^{m}\left|\sigma_j\right|^{q}\int\limits_{K}\left|\textbf{H}(a_j, c_j, g_j, \varphi)\right|^{q} d\nu(\varphi)\right]^\frac{1}{q}\nonumber \\
&\leq \delta\cdot\sup\limits_{\varphi\in K}\left[\sum\limits_{j=1}^{m} \left|\sigma_j\right|^{q} \left|\textbf{H}(a_{j}, c_{j}, g_{j}, \varphi)\right|^{q}\right]^{\frac{1}{q}}.
\end{align}
Taking the supremum over $\mu$ on $W$ on the left side of (\ref{five}) and from Proposition \ref{two}, we get
$$\mathfrak{m}_{(s;q)}^{\textbf{M}}\left((\sigma_{j}, T, a_{j}, c_j, g_{j}, \psi)\right)_{j=1}^{m}\leq \delta\cdot\sup\limits_{\varphi\in K}\left[\sum\limits_{j=1}^{m} \left|\sigma_j\right|^{q} \left|\textbf{H}(a_{j}, c_{j}, g_{j}, \varphi)\right|^{q}\right]^{\frac{1}{q}}.$$ Conversely, suppose that $T$ is $\textbf{H}$--$\textbf{M}$-$\left((s; q), q\right)$-mixing map. From Proposition  \ref{abuleldfddff} and using Condition (\textbf{VII})  we obtain $\textbf{\textit{J}}_{\mu}\circ T$ be  $\textbf{H}$--$\textbf{Q}_{2}$-abstract $q$-summing map with $\pi_{\textbf{H}\textbf{Q}_{2}, q}(\textbf{\textit{J}}_{\mu}\circ T)\leq \textbf{H}\textbf{M}_{\left((s; q), q\right)}(T)$. Hence, by using Proposition \ref{msfgsfgfgf1}, there exists a probability measure $\nu$ on $K$ such that $$\left[\int\limits_{W}\left|\textbf{M}(T, a, c, g, \psi)\right|^{s} d\mu(\psi)\right]^\frac{1}{s}\leq\left|\textbf{Q}_{2}(\textbf{\textit{J}}_{\mu}\circ T, a, c, g)\right|\leq\textbf{H}\textbf{M}_{\left((s, q), q\right)}(T)\cdot\left[\int\limits_{K}\left|\textbf{H}(a, c, g, \varphi )\right|^{q} d\nu(\varphi)\right]^\frac{1}{q}$$ 
for all $a\in A$, $c\in C$ and $g\in G$.
\end{proof} 
The  next inclusion result follows immediately from  Proposition \ref{abuleldgfgfhgfhgfddff}.
\begin{prop}
If $q_{1}\leq q_{2}\leq s_{2}\leq s_{1}$, then $$\mathfrak{M}_{\left((s_{1}; q_{1}), q_{1}\right)}^{\textbf{H}-\textbf{M}}\left(A, B\right)\subseteq\mathfrak{M}_{\left((s_{2}; q_{2}), q_{2}\right)}^{\textbf{H}-\textbf{M}}\left(A, B\right).$$
\end{prop}
\begin{prop}
Let the maps $\textbf{M}_{1}$, $\textbf{M}_{2}$, $\textbf{H}_{1}$ and $\textbf{M}$  satisfy Conditions (\textbf{V}) and (\textbf{VI}), respectively. If  $0 < p\leq s\leq t\leq\infty$, then $$\mathfrak{M}_{\left((t; s), s\right)}^{\textbf{H}_{1}-\textbf{M}_{1}}\left(B, D\right)\circ\mathfrak{M}_{\left((s; q), q\right)}^{\textbf{H}-\textbf{M}}\left(A, B\right)\subseteq\mathfrak{M}_{\left((t; q), q\right)}^{\textbf{H}-\textbf{M}_{2}}\left(A, D\right).$$
\end{prop}
\begin{proof} 
From Definition \ref{sadoon}, we have
\begin{align}
\mathfrak{m}_{(t;q)}^{\textbf{M}_{2}}\left((\sigma_{j}, S\circ T, a_{j}, c_{j}, g_{j}, \psi)\right)_{j=1}^{m}&=\inf\limits_{\tau}\left[\sum\limits_{j=1}^{m}\left|\tau_j\right|^{r}\right]^{\frac{1}{r}}\sup\limits_{\psi\in W}\left[\sum\limits_{j=1}^{m}\left|\frac{\sigma_{j}}{\tau_{j}}\right|^{t} \left|\textbf{M}_{2}(S\circ T, a_{j}, c_{j}, g_{j}, \psi)\right|^{t}\right]^{\frac{1}{t}} \nonumber \\
&=\inf\limits_{\tau_{1}\cdot\tau_{2}}\left[\sum\limits_{j=1}^{m}\left|\tau_j^{1}\cdot\tau_j^{2}\right|^{r}\right]^{\frac{1}{r}}\sup\limits_{\psi\in W}\left[\sum\limits_{j=1}^{m}\left|\frac{\sigma_{j}}{\tau_j^{1}\cdot\tau_j^{2}}\right|^{t} \left|\textbf{M}_{2}(S\circ T, a_{j}, c_{j}, g_{j}, \psi)\right|^{t}\right]^{\frac{1}{t}}. \nonumber 
\end{align}
Put $\sigma'=\frac{\sigma}{\tau_{1}}$. Since $\frac{1}{r_{1}}+\frac{1}{r_{2}}=\frac{1}{r}$ with  H\"older inequality give us
\begin{align}
&\mathfrak{m}_{(t;q)}^{\textbf{M}_{2}}\left((\sigma_{j}, S\circ T, a_{j}, c_{j}, g_{j}, \psi)\right)_{j=1}^{m}  \nonumber \\
&\leq\inf\limits_{\tau_{1}\cdot\tau_{2}}\left[\sum\limits_{j=1}^{m}\left|\tau_j^{1}\right|^{r_{1}}\right]^{\frac{1}{r_{1}}}\cdot\left[\sum\limits_{j=1}^{m}\left|\tau_j^{1}\right|^{r_{2}}\right]^{\frac{1}{r_{2}}}\cdot\sup\limits_{\psi\in W}\left[\sum\limits_{j=1}^{m}\left|\frac{\sigma_{j}}{\tau_j^{1}\cdot\tau_j^{2}}\right|^{t} \left|\textbf{M}_{1}(S, T a_{j}, c_{j}, g_{j}, \psi)\right|^{t}\right]^{\frac{1}{t}}\nonumber \\
&=\inf\limits_{\tau_{1}}\left[\sum\limits_{j=1}^{m}\left|\tau_j^{1}\right|^{r_{1}}\right]^{\frac{1}{r_{1}}}\cdot\inf\limits_{\tau_{2}}\left[\sum\limits_{j=1}^{m}\left|\tau_j^{1}\right|^{r_{2}}\right]^{\frac{1}{r_{2}}}\cdot \sup\limits_{\psi\in W}\left[\sum\limits_{j=1}^{m}\left|\frac{\sigma_{j}}{\tau_j^{1}\cdot\tau_j^{2}}\right|^{t} \left|\textbf{M}_{1}(S, T a_{j}, c_{j}, g_{j}, \psi)\right|^{t}\right]^{\frac{1}{t}} \nonumber \\
&=\inf\limits_{\tau_{1}}\left[\sum\limits_{j=1}^{m}\left|\tau_j^{1}\right|^{r_{1}}\right]^{\frac{1}{r_{1}}}\cdot\mathfrak{m}_{(t;s)}^{\textbf{M}_{1}}\left((\sigma'_j, S, T a_{j}, c_{j}, g_{j}, \psi)\right)_{j=1}^{m} \nonumber \\
&\leq\textbf{H}\textbf{M}_{\left((t; s), s\right)}(S)\inf\limits_{\tau_{1}}\left[\sum\limits_{j=1}^{m}\left|\tau_j^{1}\right|^{r_{1}}\right]^{\frac{1}{r_{1}}}\sup\limits_{\psi\in W}\left[\sum\limits_{j=1}^{m} \left|\sigma'_j\right|^{s} \left|\textbf{H}_{1}(T a_{j}, c_{j}, g_{j}, \psi)\right|^{s}\right]^{\frac{1}{s}}\nonumber \\
&\leq\textbf{H}\textbf{M}_{\left((t; s), s\right)}(S)\inf\limits_{\tau_{1}}\left[\sum\limits_{j=1}^{m}\left|\tau_j^{1}\right|^{r_{1}}\right]^{\frac{1}{r_{1}}}\sup\limits_{\psi\in W}\left[\sum\limits_{j=1}^{m}\left|\sigma'_j\right|^{s} \left|\textbf{M}(T, a_{j}, c_{j}, g_{j}, \psi)\right|^{s}\right]^{\frac{1}{s}} \nonumber \\
&=\textbf{H}\textbf{M}_{\left((t; s), s\right)}(S)\cdot\mathfrak{m}_{(s;q)}^{\textbf{M}}\left((\sigma_j, T, a_{j}, c_{j}, g_{j}, \psi)\right)_{j=1}^{m} \nonumber \\
&\leq\textbf{H}\textbf{M}_{\left((t; s), s\right)}(S)\cdot\textbf{H}\textbf{M}_{\left((s; q), q\right)}(T)\sup\limits_{\varphi\in K}\left[\sum\limits_{j=1}^{m} \left|\sigma_j\right|^{q} \left|\textbf{H}(a_{j}, c_{j}, g_{j}, \varphi)\right|^{q}\right]^{\frac{1}{q}}.\nonumber
\end{align}
\end{proof} 
\section{Recovering the known  fundamental characterizations of mixing maps} \label{IIIIII} 
\subsection{The characterizations of $(s; q)$-mixing operators}  
\begin{enumerate}
\item \cite [Theorem 20.1.4] {P78} says that a bounded operator $T$ from $E$ into $F$ is $(s; q)$-mixing  if and only if there is a constant $\delta\geq 0$ such that 
$$\Bigg[\sum\limits_{j=1}^{m}\bigg[\sum\limits_{k=1}^{n}\left|\left\langle b_{k}^{*},T x_j\right\rangle\right|^{s}\bigg]^\frac{q}{s}\Bigg]^{\frac{1}{q}}  
\leq \delta\cdot \sup\limits_{x^{*}\in B_{{E}^{*}}}\left[\sum\limits_{j=1}^{m}\left|x^{*}(x_j)\right|^{q}\right]^{\frac{1}{q}}\cdot\left[\sum\limits_{k=1}^{n}\left\|b_k^{*}\right\|^{s}\right]^{\frac{1}{s}}.$$
for every $x_{1}, \cdots,x_{m}\in E$, functional $b_{1}^{*}, \cdots, b_{n}^{*}\in F^{*}$ and $m,n\in\mathbb{N}$. We put $A:=E$, $B:=F$, $C:=\left\{1\right\}$, $G:=\mathbb{R}$, $V=F^{*}$, $W:=B_{F^{*}}$, $K:=B_{E^{*}}$, $\mathcal{H}$ be a family of bounded linear operators from $E$ into $F$, and the family of semi-norms $\mathcal{P}$ can be taken to be the single norm $\textbf{P}_{b^{*}}$ defined on $F^{*}$ by $\textbf{P}_{b^{*}}=\sup\limits_{x\in B_{E}}\left|\left\langle x, b^{*}\right\rangle\right|$. We define the maps as follows.
\begin{align}
&\textbf{M}:\mathcal{H}\times E\times \left\{1\right\}\times \mathbb{R}\times B_{F^{*}}\longrightarrow \mathbb{R},\  \textbf{M}(T, x, 1, \sigma, b^{*})=\frac{\left\langle T x, b^{*}\right\rangle}{\sigma}, \nonumber \\
&\textbf{H}: E\times \left\{1\right\}\times \mathbb{R}\times B_{E^{*}}\longrightarrow \mathbb{R},\  \textbf{H}(x, 1, \sigma, x^{*})=\frac{\left\langle  x, x^{*}\right\rangle}{\sigma}, \nonumber 
\end{align}
where $\sigma\neq 0$. With these choices and applying Proposition \ref{aaa} we obtain $T$ be $(s; q)$-mixing operator if and only if $T$ be \textbf{H}--\textbf{M}-$\left((s; q), q\right)$-mixing operator. In this context Proposition \ref{aaa} coincides with Theorem 20.1.4 in \cite{P78} for  $(s; q)$-mixing operator. 
\item \cite [Theorem 20.1.7] {P78} says that a bounded operator $T$ from $E$ into $F$ is $(s; q)$-mixing   if and only if there is a constant $\delta\geq 0$ such that for any probability measure $\mu$ on $B_{F^{*}}$ there exists a probability measure $\nu$ on $B_{E^{*}}$ such that 
$$\left[\int\limits_{B_{F^{*}}}\left|\left\langle T x, b^{*}\right\rangle\right|^{s} d\mu(b^{*})\right]^\frac{1}{s}\leq\delta\cdot\left[\int\limits_{B_{E^{*}}}\left|\left\langle x, x^{*}\right\rangle\right|^{q} d\nu(x^{*})\right]^\frac{1}{q},$$
whenever $x\in E$.  We define the maps as follows.
\begin{align}
&\textbf{M}:\mathcal{H}\times E\times \left\{1\right\}\times \mathbb{R}\times B_{F^{*}}\longrightarrow \mathbb{R},\  \textbf{M}(T, x, 1, \sigma, b^{*})=\left\langle T x, b^{*}\right\rangle, \nonumber \\
&\textbf{H}: E\times \left\{1\right\}\times \mathbb{R}\times B_{E^{*}}\longrightarrow \mathbb{R},\  \textbf{H}(x, 1, \sigma, x^{*})=\left\langle  x, x^{*}\right\rangle. \nonumber 
\end{align}
With the above choices and applying Proposition \ref{abuleldgfgfhgfhgfddff} we have $T$ be $(s; q)$-mixing operator if and only if $T$ be \textbf{H}--\textbf{M}-$\left((s; q), q\right)$-mixing operator. In this context Proposition \ref{abuleldgfgfhgfhgfddff} coincides with Theorem 20.1.7 in \cite{P78} for  $(s; q)$-mixing operator.  
\end{enumerate}
\subsection{The characterizations of Lipschitz $(s; q)$-mixing maps}  
\begin{enumerate}
\item  \cite [Theorem 4.4] {S16} says that a Lipschitz map $T$ from $X$ into $Y$ is Lipschitz $(s; q)$-mixing   if and only if there is a constant $\delta\geq 0$ such that 
\begin{align}
&\Bigg[\sum\limits_{j=1}^{m}\left|\sigma_j\right|^{q}\bigg[\sum\limits_{k=1}^{n}\left|\left\langle g_{k}, Tx'_j\right\rangle_{(Y^{\#},Y)}- \left\langle   g_{k}, Tx''_j\right\rangle_{(Y^{\#},Y)}\right|^{s}\bigg]^\frac{q}{s}\Bigg]^{\frac{1}{q}} \nonumber \\
&\leq C\cdot\sup\limits_{f\in B_{{X}^{\#}}}\Bigg[\sum\limits_{j=1}^{m}\left|\sigma_j\right|^{q}\left|fx'_j-fx''_j\right|^{q}\Bigg]^{\frac{1}{q}}
\Bigg[\sum\limits_{k=1}^{n}\Lip(g_k)^{s}\Bigg]^{\frac{1}{s}} \nonumber
\end{align}
for every nonzero $\sigma_1,\cdot\cdot\cdot,\sigma_m\in\mathbb{R}$, $x'_1,\cdots, x'_m, x''_1,\cdots,x''_m\in X$, $g_1,\cdots,g_n\in Y^{\#}$ and $m,n\in\mathbb{N}$. We put $A:=X$, $B:=Y$, $C:=X$, $G:=\mathbb{R}$, $V:=Y^{\#}$, $W:=B_{Y^{\#}}$, $K:=B_{X^{\#}}$, $\mathcal{H}$ be a family of Lipschitz maps from $X$ into $Y$, and the family of semi-norms $\mathcal{P}$ can be taken to be the single norm $\textbf{P}_{\tilde{g}}$ defined on $Y^{\#}$ by $\textbf{P}_{\tilde{g}}=\sup\limits_{x'\neq x''}\frac{\left|\tilde{g} x'- \tilde{g} x''\right|}{d_{X}\left(x',x''\right)}$. We define the maps as follows.
\begin{align}
&\textbf{M}:\mathcal{H}\times  X\times X\times\mathbb{R}\times B_{Y^{\#}}\longrightarrow \mathbb{R},\  \textbf{M}(T, x', x'',   \sigma, \tilde{g})=\left\langle \tilde{g}, Tx'\right\rangle_{(Y^{\#},Y)}- \left\langle   \tilde{g}, Tx''\right\rangle_{(Y^{\#},Y)}, \nonumber \\
&\textbf{H}: X\times X\times \mathbb{R}\times B_{X^{\#}}\longrightarrow \mathbb{R},\  \textbf{H}(x', x'',  \sigma, f)=\left\langle f, x'\right\rangle_{(X^{\#},X)}- \left\langle   f, x''\right\rangle_{(X^{\#},X)}. \nonumber 
\end{align}
With these choices and applying Proposition \ref{aaa} we get $T$ is Lipschitz $(s; q)$-mixing map if and only if $T$ is \textbf{H}--\textbf{M}-$\left((s; q), q\right)$-mixing map. In this context Proposition \ref{aaa} coincides with Theorem 4.4 in \cite{S16} for Lipschitz $(s; q)$-mixing map. 
\item \cite [Theorem 4.1] {C12} says that a Lipschitz map $T$ from $X$ into $Y$ is Lipschitz $(s; q)$-mixing   if and only if there is a constant $\delta\geq 0$ such that for any probability measure $\mu$ on $B_{Y^{\#}}$ there exists a probability measure $\nu$ on $B_{X^{\#}}$ such that $$\left[\int\limits_{B_{Y^{\#}}}\left|\left\langle g, Tx'\right\rangle_{(Y^{\#},Y)}- \left\langle   g, Tx''\right\rangle_{(Y^{\#},Y)}\right|^{s} d\mu(g)\right]^\frac{1}{s}\leq\delta\cdot\left[\int\limits_{B_{X^{\#}}}\left|\left\langle f, x'\right\rangle_{(X^{\#},X)}- \left\langle   f, x''\right\rangle_{(X^{\#},X)}\right|^{q} d\nu(f)\right]^\frac{1}{q},$$
whenever $x'_1,\cdots, x'_m, x''_1,\cdots,x''_m\in X$, and $m\in\mathbb{N}$. With the above  choices and applying Proposition \ref{abuleldgfgfhgfhgfddff} we have $T$ is Lipschitz $(s; q)$-mixing map if and only if $T$ is \textbf{H}--\textbf{M}-$\left((s; q), q\right)$-mixing map. In this context  Proposition \ref{abuleldgfgfhgfhgfddff} coincides with Theorem 4.1 in \cite{C12} for Lipschitz $(s; q)$-mixing map. 
\end{enumerate}
\section{ R$_1$, ..., R$_t$-S-((s, q), p$_1$, ..., p$_t$)-mixing maps}  \label{IIIIIII}
Let $A_1,\cdots, A_t$, $B$ and $C_1,\cdots, C_r$ be non-void sets, $\mathcal{H}$ be a non-void family of mappings from $A_1\times\cdots\times A_t$ into $B$, and $G_1,\cdots, G_s$ be  Banach spaces. Let  $W$ and $K_1,\cdots, K_s$ be compact Hausdorff topological spaces. Put $\widetilde{A}:= A_1\times\cdots\times A_t$, $\widetilde{C}:= C_1\times\cdots\times C_r$ and $\widetilde{G}:= G_1\times\cdots\times G_s$. 

Let $\textbf{M}:\mathcal{H}\times \widetilde{A}\times\widetilde{C}\times  \widetilde{G}\times W\longrightarrow \mathbb{R}$ and $\textbf{H}_{k}: \widetilde{A}\times\widetilde{C}\times G_{k}\times K_{k}\longrightarrow \mathbb{R}, \; k=1,\ldots, s$ be arbitrary maps satisfy the following conditions:
  
\textbf{(VIII)}  The mapping $\textbf{M}_{T, a_1,\ldots, a_t, c_1,\ldots, c_r, g_1,\ldots, g_s}: W\longrightarrow \mathbb{R}$ defined by $$\textbf{M}_{T, a_1,\ldots, a_t, c_1,\ldots, c_r, g_1,\ldots, g_s}(\psi)=\textbf{M}(T, a_1,\ldots, a_t, c_1,\ldots, c_r, g_1,\ldots, g_s, \psi)$$ is continuous for every $T\in\mathcal{H}$, $a_1,\ldots, a_t\in \widetilde{A}$, $c_1,\ldots, c_r\in \widetilde{C}$ and $g_1,\ldots, g_s\in \widetilde{G}$.

\textbf{(VIIII)}  The mapping $\textbf{M}$ be a homogeneous of degree $1$ in the variable $W$ if $$\textbf{M}(T, a_1,\ldots, a_t, c_1,\ldots, c_r, g_1,\ldots, g_s, \lambda \;\psi)= \lambda\;\textbf{M}(T, a_1,\ldots, a_t, c_1,\ldots, c_r, g_1,\ldots, g_s,  \psi).$$

\textbf{(X)}  The mapping $(\textbf{H}_{k})_{a_1,\ldots, a_t, c_1,\ldots, c_r, g}: K_k\longrightarrow \mathbb{R}$ defined by    $$(\textbf{H}_{k})_{a_1,\ldots, a_t, c_1,\ldots, c_r, g}(\varphi)=\textbf{H}_{k}(a_1,\ldots, a_t, c_1,\ldots, c_r, g, \varphi)$$ is continuous for every  $a_1,\ldots, a_t\in \widetilde{A}$, $c_1,\ldots, c_r\in \widetilde{C}$ and  $g\in G_k$. \\ The concept of \textbf{H}$_1$,..., \textbf{H}$_t$-\textbf{M}-$\left((s; q), p_1, ..., p_t\right)$-mixing map can be constructed as follows.
\begin{defi}
Let $0<q\leq s\leq\infty$ and $p\leq q$. A map $T$ from $\widetilde{A}$ into $B$ is called \textbf{H}$_1$,..., \textbf{H}$_t$-\textbf{M}-$\left((s; q), p_1, ..., p_t\right)$-mixing  if there is a constant $\delta\geq 0$ such that 
\begin{align}
&\mathfrak{m}_{(s;q)}^{\textbf{M}}\left((\sigma_{j}, T, a_{j}^{1}, \ldots, a_{j}^{t}, c_{j}^{1}, \ldots, c_{j}^{r}, g_{j}^{1}, \ldots ,g_{j}^{s}, \psi)\right)_{j=1}^{m} \nonumber \\
&\leq \delta\cdot\displaystyle\prod_{k=1}^{s}\sup\limits_{\varphi\in K_{k}}\left[\sum\limits_{j=1}^{m} \left|\sigma_j\right|^{p_{k}} \left|\textbf{H}_{k}(a_{j}^{1}, \ldots, a_{j}^{t}, c_{j}^{1}, \ldots, c_{j}^{r}, g_{j}^{k}, \varphi)\right|^{p_{k}}\right]^{\frac{1}{p_{k}}}
\end{align}
for all nonzero $\sigma_{1}, \cdots,\sigma_{m}\in\mathbb{R}$, $a_1,\ldots, a_t\in \widetilde{A}$, $c_{1}^{1}, \cdots, c_{m}^{r}\in \widetilde{C}$, $g_{1}^{1}, \cdots, g_{m}^{s}\in \widetilde{G}$ and $m\in\mathbb{N}$. The infimum of such constants $\delta$ is denoted by \textbf{H}$_1$,..., \textbf{H}$_t$-\textbf{M}$_{\left((s; q), p_1, ..., p_t\right)}(T)$. Let us denote by $\mathfrak{M}_{\left((s; q), p_1, ..., p_t\right)}^{\textbf{H}_1,..., \textbf{H}_t-\textbf{M}}\left(\widetilde{A}, B\right)$ the class of all \textbf{H}$_1$,..., \textbf{H}$_t$-\textbf{M}-$\left((s; q), p_1, ..., p_t\right)$-mixing maps from $\widetilde{A}$ into $B$.
\end{defi}
\begin{prop}\label{a} 
Let $0<q\leq s\leq\infty$ and $p\leq q$ and let $\textbf{H}_{k}$ and \textbf{M} satisfy Conditions (\textbf{X}), (\textbf{VIII}) and (\textbf{VIIII}), respectively. A map $T$ from $\widetilde{A}$ into $B$ is \textbf{H}$_1$,..., \textbf{H}$_t$-\textbf{M}-$\left((s; q), p_1, ..., p_t\right)$-mixing if and only if there is a constant $\delta\geq 0$ such that 
\begin{align}\label{tytytyi}
&\Bigg[\sum\limits_{j=1}^{m}\left|\sigma_j\right|^{q}\bigg[\sum\limits_{\zeta=1}^{n}\left|\textbf{M}(T, a_{j}^{1}, \ldots, a_{j}^{t}, c_{j}^{1},\ldots, c_{j}^{r}, g_{j}^{1},\ldots, g_{j}^{s}, v_{\zeta})\right|^{s}\bigg]^\frac{q}{s}\Bigg]^{\frac{1}{q}} \nonumber \\
&\leq\delta\cdot\displaystyle\prod_{k=1}^{s}\sup\limits_{\varphi\in K_{k}}\left[\sum\limits_{j=1}^{m}\left|\sigma_j\right|^{p_{k}} \left|\textbf{H}_{k}(a_{j}^{1}, \ldots, a_{j}^{t}, c_{j}^{1}, \ldots, c_{j}^{r}, g_{j}^{k}, \varphi)\right|^{p_{k}}\right]^{\frac{1}{p_{k}}} 
\cdot\left[\sum\limits_{\zeta=1}^{n} \textbf{P}(v_{\zeta})^{s}\right]^{\frac{1}{s}}  
\end{align} 
for every nonzero $\sigma_{1}, \cdots,\sigma_{m}\in\mathbb{R}$, $a_{j}^{1}, \ldots, a_{j}^{t}\in \widetilde{A}$, $c_{1}^{1}, \cdots, c_{m}^{r}\in \widetilde{C}$, $g_{1}^{1}, \cdots, g_{m}^{s}\in \widetilde{G}$, $v_1,\cdots, v_n\in V$ and $m,n\in\mathbb{N}$. Moreover 
$$\textbf{H}_1,..., \textbf{H}_t-\textbf{M}_{\left((s; q), p_1, ..., p_t\right)}(T)=\inf \delta.$$
\end{prop} 
\begin{proof} There are two cases. 

\textbf{Case (1)}: when $s=q$. Assume that Inequality (\ref{tytytyi}) holds and take $\zeta=1$ we have 
\begin{align}
&\Bigg[\sum\limits_{j=1}^{m}\left|\sigma_j\right|^{q}\left|\textbf{M}(T, a_{j}^{1}, \ldots, a_{j}^{t}, c_{j}^{1},\ldots, c_{j}^{r}, g_{j}^{1},\ldots, g_{j}^{s}, v)\right|^{q}\Bigg]^{\frac{1}{q}}  \nonumber \\
&\leq\delta\cdot\displaystyle\prod_{k=1}^{s}\sup\limits_{\varphi\in K_{k}}\left[\sum\limits_{j=1}^{m} \left|\sigma_j\right|^{p_{k}} \left|\textbf{H}_{k}(a_{j}^{1}, \ldots, a_{j}^{t}, c_{j}^{1}, \ldots, c_{j}^{r}, g_{j}^{k}, \varphi)\right|^{p_{k}}\right]^{\frac{1}{p_{k}}}   \nonumber
\end{align}
for all $v\in\mathfrak{B}_{\textbf{P}}$. Hence 
\begin{align}\label{bfgus}
&\sup\limits_{v\in\mathfrak{B}_{\textbf{P}}}\left[\sum\limits_{j=1}^{m}\left| \sigma_{j}\right|^{q} \left|\textbf{M}(T, a_{j}^{1}, \ldots, a_{j}^{t}, c_{j}^{1},\ldots, c_{j}^{r}, g_{j}^{1},\ldots, g_{j}^{s}, v)\right|^{q}\right]^{\frac{1}{q}}  \nonumber \\
&\leq\delta\cdot\displaystyle\prod_{k=1}^{s}\sup\limits_{\varphi\in K_{k}}\left[\sum\limits_{j=1}^{m}\left|\sigma_j\right|^{p_{k}} \left|\textbf{H}_{k}(a_{j}^{1}, \ldots, a_{j}^{t}, c_{j}^{1}, \ldots, c_{j}^{r}, g_{j}^{k}, \varphi)\right|^{p_{k}}\right]^{\frac{1}{p_{k}}}.  
\end{align} 
From Lemma \ref{twotwo} and using Inequality (\ref{bfgus}) we obtain
\begin{align} 
&\mathfrak{m}_{(s;q)}^{\textbf{M}}\left((\sigma_{j}, T, a_{j}^{1}, \ldots, a_{j}^{1}, \ldots, a_{j}^{t}, c_{j}^{r}, g_{j}^{1}, \ldots ,g_{j}^{s}, v)\right)_{j=1}^{m}  \nonumber \\
& =\sup\limits_{v\in\mathfrak{B}_{\textbf{P}}}\left[\sum\limits_{j=1}^{m}\left| \sigma_{j}\right|^{q} \left|\textbf{M}(T, a_{j}^{1}, \ldots, a_{j}^{t}, c_{j}^{1},\ldots, c_{j}^{r}, g_{j}^{1},\ldots, g_{j}^{s}, v)\right|^{q}\right]^{\frac{1}{q}}  \nonumber \\
&\leq\delta\cdot\displaystyle\prod_{k=1}^{s}\sup\limits_{\varphi\in K_{k}}\left[\sum\limits_{j=1}^{m}\left|\sigma_j\right|^{p_{k}} \left|\textbf{H}_{k}(a_{j}^{1}, \ldots, a_{j}^{t}, c_{j}^{1}, \ldots, c_{j}^{r}, g_{j}^{k}, \varphi)\right|^{p_{k}}\right]^{\frac{1}{p_{k}}}. \nonumber 
\end{align} 
Hence $T$ is \textbf{H}$_1$,..., \textbf{H}$_t$-\textbf{M}-$\left((s; q), p_1, ..., p_t\right)$-mixing map and $\textbf{H}_1,..., \textbf{H}_t-\textbf{M}_{\left((s; q), p_1, ..., p_t\right)}(T)\leq \delta$. Conversely, suppose that $T$ is \textbf{H}$_1$,..., \textbf{H}$_t$-\textbf{M}-$\left((s; q), p_1, ..., p_t\right)$-mixing map. Given $\sigma_{1}, \cdots,\sigma_{m}\in\mathbb{R}$, $a_{1}^{1}, \cdots, a_{m}^{t}\in \widetilde{A}$, $c_{1}^{1}, \cdots, c_{m}^{r}\in \widetilde{C}$, $g_{1}^{1}, \cdots, g_{m}^{s}\in \widetilde{G}$, $v_1,\cdots, v_n\in V$ and $m,n\in\mathbb{N}$. 
\begin{align}\label{kamoon}
&\Bigg[\sum\limits_{j=1}^{m}\left|\sigma_j\right|^{q}\sum\limits_{\zeta=1}^{n}\left|\textbf{M}(T, a_{j}^{1}, \ldots, a_{j}^{t}, c_{j}^{1},\ldots, c_{j}^{r}, g_{j}^{1},\ldots, g_{j}^{s}, v_{\zeta})\right|^{q}\Bigg]^{\frac{1}{q}} \nonumber \\
&\leq\left\|\left(\tau_{j}\right)_{j=1}^{m}\Big|\ell_{\infty}\right\|\cdot\Bigg[\sum\limits_{j=1}^{m}\left|\frac{\sigma_j}{\tau_{j}}\right|^{q}\sum\limits_{\zeta=1}^{n}\textbf{P}(v_{\zeta})^{q}\left|\textbf{M}(T, a_{j}^{1}, \ldots, a_{j}^{t}, c_{j}^{1},\ldots, c_{j}^{r}, g_{j}^{1},\ldots, g_{j}^{s}, \frac{v_{\zeta}}{\textbf{P}(v_{\zeta})})\right|^{q}\Bigg]^{\frac{1}{q}} \nonumber \\
&\leq\left\|\left(\tau_{j}\right)_{j=1}^{m}\Big|\ell_{\infty}\right\|\sup\limits_{v\in\mathfrak{B}_{\textbf{P}}}\left[\sum\limits_{j=1}^{m}\left|\frac{\sigma_j}{\tau_{j}}\right|^{q} \left|\textbf{M}(T, a_{j}^{1}, \ldots, a_{j}^{t}, c_{j}^{1},\ldots, c_{j}^{r}, g_{j}^{1},\ldots, g_{j}^{s}, v)\right|^{q}\right]^{\frac{1}{q}}\left[\sum\limits_{\zeta=1}^{n}\textbf{P}(v_{\zeta})^{q}\right]^{\frac{1}{q}}.  
\end{align}  
Taking the infimum over all sequences $\left(\tau_{j}\right)_{j=1}^{m}\in \ell_{\infty}$ of Inequality (\ref{kamoon}) we get
\begin{align}
&\Bigg[\sum\limits_{j=1}^{m}\left|\sigma_j\right|^{q}\sum\limits_{\zeta=1}^{n}\left|\textbf{M}(T, a_{j}^{1}, \ldots, a_{j}^{t}, c_{j}^{1},\ldots, c_{j}^{r}, g_{j}^{1},\ldots, g_{j}^{s}, v_{\zeta})\right|^{q}\Bigg]^{\frac{1}{q}} \nonumber \\
&\leq\left[\sum\limits_{\zeta=1}^{n}\textbf{P}(v_{\zeta})^{q}\right]^{\frac{1}{q}}\mathfrak{m}_{(q;q)}^{\textbf{M}}\left((\sigma_{j}, T, a_{j}^{1}, \ldots, a_{j}^{t}, c_{j}^{1}, \ldots, c_{j}^{r}, g_{j}^{1}, \ldots ,g_{j}^{s}, \psi)\right)_{j=1}^{m}\nonumber \\
&\leq\left[\sum\limits_{\zeta=1}^{n}\textbf{P}(v_{\zeta})^{q}\right]^{\frac{1}{q}}\textbf{H}_1,..., \textbf{H}_t-\textbf{M}_{\left((q; q), p_1, ..., p_t\right)}(T)\displaystyle\prod_{k=1}^{s}\sup\limits_{\varphi\in K_{k}}\left[\sum\limits_{j=1}^{m} \left|\sigma_j\right|^{p_{k}} \left|\textbf{H}_{k}(a_{j}^{1}, \ldots, a_{j}^{t}, c_{j}^{1}, \ldots, c_{j}^{r}, g_{j}^{k}, \varphi)\right|^{p_{k}}\right]^{\frac{1}{p_{k}}}. \nonumber  
\end{align} 
\textbf{Case (2)}: when $q < s$. Suppose that $T$ is \textbf{H}$_1$,..., \textbf{H}$_t$-\textbf{M}-$\left((s; q), p_1, ..., p_t\right)$-mixing map.  
Given $\sigma_{1}, \cdots,\sigma_{m}\in\mathbb{R}$, $a_{1}^{1}, \cdots, a_{m}^{t}\in \widetilde{A}$, $c_{1}^{1}, \cdots, c_{m}^{r}\in \widetilde{C}$, $g_{1}^{1}, \cdots, g_{m}^{s}\in \widetilde{G}$. From Proposition \ref{two} and condition (\textbf{VIIII}) we have
\begin{align}
\Bigg[\sum\limits_{j=1}^{m}\left|\sigma_j\right|^{q}&\bigg[\sum\limits_{\zeta=1}^{n}\left|\textbf{M}(T, a_{j}^{1}, \ldots, a_{j}^{t}, c_{j}^{1},\ldots, c_{j}^{r}, g_{j}^{1},\ldots, g_{j}^{s}, v_{\zeta})\right|^{s}\bigg]^\frac{q}{s}\Bigg]^{\frac{1}{q}} \nonumber \\
& =\Bigg[\sum\limits_{j=1}^{m}\left|\sigma_j\right|^{q}\bigg[\int\limits_{\mathfrak{B}_{\textbf{P}}}\left|\textbf{M}(T, a_{j}^{1}, \ldots, a_{j}^{t}, c_{j}^{1},\ldots, c_{j}^{r}, g_{j}^{1},\ldots, g_{j}^{s}, v)\right|^{s} d\mu(v)\bigg]^\frac{q}{s}\Bigg]^{\frac{1}{q}}\cdot\left[\sum\limits_{\zeta=1}^{n} \textbf{P}(v_{k})^{s}\right]^{\frac{1}{s}}  \nonumber \\ 
&\leq\mathfrak{m}_{(s;q)}^{\textbf{M}}\left((\sigma_{j}, T, a_{j}^{1}, \ldots, a_{j}^{t}, c_{j}^{1}, \ldots, c_{j}^{r}, g_{j}^{1}, \ldots ,g_{j}^{s}, v)\right)_{j=1}^{m}\cdot\left[\sum\limits_{\zeta=1}^{n} \textbf{P}(v_{k})^{s}\right]^{\frac{1}{s}}\nonumber \\
&\leq\textbf{H}_1,..., \textbf{H}_t-\textbf{M}_{\left((s, q), p_1, ..., p_t\right)}(T)\cdot\left[\sum\limits_{\zeta=1}^{n} \textbf{P}(v_{k})^{s}\right]^{\frac{1}{s}}   \nonumber \\
&\cdot\displaystyle\prod_{k=1}^{s}\sup\limits_{\varphi\in K_{k}}\left[\sum\limits_{j=1}^{m} \left|\sigma_j\right|^{p_{k}} \left|\textbf{H}_{k}(a_{j}^{1}, \ldots, a_{j}^{t}, c_{j}^{1}, \ldots, c_{j}^{r}, g_{j}^{k}, \varphi)\right|^{p_{k}}\right]^{\frac{1}{p_{k}}}. \nonumber
\end{align} 
To show the converse, notice that (\ref{tytytyi}) means  
\begin{align}\label{oneone}
\Bigg[\sum\limits_{j=1}^{m}\left|\sigma_j\right|^{q}&\bigg[\int\limits_{\mathfrak{B}_{\textbf{P}}}\left|\textbf{M}(T, a_{j}^{1}, \ldots, a_{j}^{t}, c_{j}^{1},\ldots, c_{j}^{r}, g_{j}^{1},\ldots, g_{j}^{s}, v)\right|^{s} d\mu(v)\bigg]^\frac{q}{s}\Bigg]^{\frac{1}{q}}\nonumber \\
&\leq \delta\cdot\displaystyle\prod_{k=1}^{s}\sup\limits_{\varphi\in K_{k}}\left[\sum\limits_{j=1}^{m}\left|\sigma_j\right|^{p_{k}} \left|\textbf{H}_{k}(a_{j}^{1}, \ldots, a_{j}^{t}, c_{j}^{1}, \ldots, c_{j}^{r}, g_{j}^{k}, \varphi)\right|^{p_{k}}\right]^{\frac{1}{p_{k}}}  
\end{align}
for every discrete probability measure $\mu$ on $\mathfrak{B}_{\textbf{P}}$ and $\sigma_{1}, \cdots,\sigma_{m}\in\mathbb{R}$, $a_{1}^{1}, \cdots, a_{m}^{t}\in \widetilde{A}$, $c_{1}^{1}, \cdots, c_{m}^{r}\in \widetilde{C}$, $g_{1}^{1}, \cdots, g_{m}^{s}\in \widetilde{G}$. It follows that (\ref{oneone}) holds for all probability measures $\mu$ on $\mathfrak{B}_{\textbf{P}}$ and $\sigma_{1}, \cdots,\sigma_{m}\in\mathbb{R}$, $a_{1}^{1}, \cdots, a_{m}^{t}\in \widetilde{A}$, $c_{1}^{1}, \cdots, c_{m}^{r}\in \widetilde{C}$, $g_{1}^{1}, \cdots, g_{m}^{s}\in \widetilde{G}$. Taking the supremum over $\mu\in \textbf{W}(\mathfrak{B}_{\textbf{P}})$ on the left side of (\ref{oneone}) and using Proposition \ref{two}, we obtain
\begin{align}
&\mathfrak{m}_{(s;q)}^{\textbf{M}}\left((\sigma_{j}, T, a_{j}^{1}, \ldots, a_{j}^{t}, c_{j}^{1}, \ldots, c_{j}^{r}, g_{j}^{1}, \ldots ,g_{j}^{s}, \psi)\right)_{j=1}^{m}\nonumber \\
&\leq \delta\cdot\displaystyle\prod_{k=1}^{s}\sup\limits_{\varphi\in K_{k}}\left[\sum\limits_{j=1}^{m}\left|\sigma_j\right|^{p_{k}} \left|\textbf{H}_{k}(a_{j}^{1}, \ldots, a_{j}^{t}, c_{j}^{1}, \ldots, c_{j}^{r}, g_{j}^{k}, \varphi)\right|^{p_{k}}\right]^{\frac{1}{p_{k}}}. \nonumber 
\end{align}
\end{proof}
           
\end{document}